\documentclass{amsart}
\date{December 2, 2011}
\usepackage{pst-node}
\usepackage{color}

\usepackage{color}
\newtheorem{Theorem}{Theorem}[section]

\newtheorem{Lemma}[Theorem]{Lemma}

\newtheorem{Corollary}[Theorem]{Corollary}
\newenvironment{The}[1][Theorem]{\begin{trivlist} \item[\hskip \labelsep {\bfseries #1.}]\it }{\end{trivlist}}

\newcommand{\G}     {\mbox{$\mathcal G$}}

\begin{document}

\title{All totally symmetric colored graphs}
\author{Mariusz Grech, Andrzej Kisielewicz}
\address{Institute of Mathematics, University of Wroc{\l}aw \
pl.Grunwaldzki 2, 50-384 Wroc{\l}aw, Poland}
\email{Mariusz.Grech@math.uni.wroc.pl, Andrzej.Kisielewicz@math.uni.wroc.pl}

\begin{abstract}

In this paper we describe all edge-colored graphs that are fully
symmetric with respect to colors and transitive on every set of edges of the
same color. They correspond to fully symmetric homogeneous factorizations of complete graphs. Our description completes the work done in our previous paper, where we have shown, in particular, that there are no such graphs with more than 5 colors. Using some recent results, with a help of computer, we settle all the cases that was left open in the previous paper.

\end{abstract}

\maketitle

%%\section{Result}

A $k$-colored graph $\G = (V,\psi)$ on a set of vertices $V$ is a complete
graph on $V$ together with a function $\psi$ from the set of edges \emph{onto}
the set of colors $\{0,1,\ldots,k-1\}$. The automorphism group $Aut(\G$) of
$\G$ is the set of permutations of $V$ preserving the colors of the edges. The
\emph{extended automorphism group} $Ext(\G$) is the set of permutations of $V$
preserving the partition into the colors. Obviously, $Aut(\G$) is a normal
subgroup of $Ext(\G$). Moreover, the factor group ${Ext}(\G)/{Aut}(\G)$ may be
considered as one acting on the set of colors, and as such is called the
\emph{symmetry group of colors} of $\G$.

A graph $\G$ is called \emph{edge-transitive} if $Aut(\G)$ acts transitively on
each set of the edges of the same color, and is called \emph{arc-transitive} (or \emph{strongly edge-transitive}) if $Aut(\G)$ acts transitively on each set of ordered pairs of vertices
corresponding to a set of edges of the same color. It is called
\emph{color-symmetric}, if $Ext(\G)$ acts as the symmetric group on the set of
colors. If $\G$ is both color-symmetric and edge-transitive it is called \emph{totally symmetric}, or TSC-\emph{graph}, in short. In the previous paper \cite{tsc1}, we have proved that TSC-graphs may occur only for less than 6 colors. In this paper we give a complete description of such graphs. For two colors these are self-complementary
symmetric graphs, which have been recently described by Peisert
\cite{Peisert}.  Our result may be viewed as a natural generalization of
Peisert's result. 

The problems closely related to the subject of this paper have been studied under various names.
In \cite{Sib}, T. Sibley classifies edge colored graphs with 2-transitive extended automorphism group $Ext(G)$ (he calls this group the \emph{automorphism group}, and only mention $Aut(G)$ calling it the \emph{group of isometries}). Colored graphs with transitive $Aut(G)$ have been considered by Chen and Teh in \cite {CT}; they call them \emph{point-color symmetric graphs}. Ashbacher \cite{Asch} uses the name \emph{rainbows} for such structures. Various highly-symmetrical structures in this class have been intensively studied using results based on the classification of finite simple groups (see e.g. \cite{Cam,CK,Gu1, Kantor,LiP,Muz,PS}. 

The most interesting colored graphs arise from \emph{factorizations of complete graphs},  
that is partitions of the edges into factors that correspond to spanning subgraphs (not necessarily connected). \emph{Isomorphic factorizations}, coloring the edges of a graph so that the colored subgraphs are isomorphic, have been introduced and studied in a series of papers by Harary, Robinson and Wormald (cf. \cite{if1,if2}). A factorization a complete graph is \emph{homogeneous} if there are subgroups $M < G \leq S_n$ such that $M$ is vertex-transitive and fixes each factor setwise, and $G$ permutes the factors transitively (in particular, the factors are isomorphic). TSC-graphs correspond to those factorizations where $M$ is edge-transitive and $G$ permutes factors symmetrically. Recently, in \cite{praeger}, Li, Lim and Praeger  have classified all homogeneous factorizations with $M$ being edge-transitive. Their result helped us to finish our study and get the following complete description of totally symmetric colored graphs.

\begin{The}\label{main}
If $\G$ is an edge-transitive color-symmetric
$k$-colored graph, then $\G$ is arc-transitive and $k \leq 5$.
Moreover, one of the following cases holds:
\begin{enumerate}
\item[\rm (i)] $k=5$ and $\G= {\mathcal F}_5(4^2)$ or $\G = {\mathcal H}_5(3^4)$;
\item[\rm (ii)] $k=4$ and $\G = {\mathcal F}_4(3^2)$;
\item[\rm (iii)] $k=3$ and $\G$ belongs to an infinite family of generalized Paley graphs $GP_3(q)$ or $\G={\mathcal G}_3(5^2)$ or $\G={\mathcal G}_3(11^2)$;
\item[\rm (iv)] $k=2$ and $\G$ belongs to an infinite family of Paley graphs $PG(q)$ or Peisert graphs $PG^*(q)$, or else $\G=\G(23^2)$.
 \end{enumerate}
 \end{The}

All graphs mentioned in the theorem are defined in the next section.

We note that this description is something much stronger than just a classification. In a sense one may say that our description is contained in the results classifying finite simple groups, rank 3-groups, and homogeneous factorizations. A good example of what makes a difference here is that (as we shall see in the sequel) all classifications show the possibility of existence of a $5$-colored TSC-graph on $2^8$ vertices. Only combining a suitable knowledge with heavy computations has allowed to learn that such an object, in fact, does not exists. 

The paper is organized as follows. In Section~\ref{sdef} we classify and give defintions of totally symmetric graphs. In Section~\ref{sproof}, we summarize and recall results needed in our proof. Then, in the three following sections we consider $4$-, $5$- and $3$-colored TSC-graphs, respectively,  which are the cases (ii), (i), and (iii) of the theorem above. The  2-colored TSC-graphs, as we have already mentioned, are described completely in \cite{Peisert}. Finally, in Section~\ref{scomp} we present computations that allowed to settle the most complicated cases that arose during consideration in Section~\ref{s5}.

\section{Definitions of TSC-graphs}\label{sdef} 

We assume that the reader is familiar with general terminology of finite fields, vector spaces and permutation groups (as used, e.g., in \cite{PC,DM}). We start from simple graphs ($k=2$). Whenever we consider a finite field $F_q$ ($q = p^r$), by $\omega$ we denote a fixed primitive root of $F_q$.

\subsection{Simple graphs.} 

Totally symmetric  2-colored graphs are simple graphs that are symmetric and self-complementary. They have been fully described in \cite{Peisert}. 
Recall first that the \emph{Paley graph} $PG(q)$, where $q = 1(\bmod~4)$, is one whose vertex
set is $F_{q}$, and two distinct elements are adjacent, if their difference is a
square in $F_{q}$ (i.e., if it is of the form $\omega^{2k}$). If in addition,
$p= 3(\bmod~4)$ (and consequently, $r$ is even), then the \emph{Peisert
graph} $ PG^*(q) $ is one whose vertex set is $F_{q}$, and two elements are
adjacent, if their difference is of the form $\omega^j$, where $j= 0,1
(\bmod~4)$. The constructions above do not depend on the choice of the primitive root.
Moreover, $ PG(q) $ is isomorphic with $ PG^*(q)$ only if $q = 9$.
For details see \cite{Peisert}.

An exceptional self-complementary symmetric graph $G(23^2)$ has a rather complicated definition. It can be found in \cite{Peisert}.

\subsection{Generalized Paley graphs} 

Following the definition of Paley graphs we define now their 3-colored counterparts.
Let $p= 2(\bmod~3)$, $q=p^r$, and  $q= 1(\bmod~3)$ (which is equivalent to $r$ being even).
Let $GP_3(q)$ denote 3-colored graph whose vertex set is $F_{q}$, and the edge between two distinct elements has color $i$ ($i=0,1,2$) if their difference is of the form $\omega^{3m+i}$.
It is not difficult to see that this definition does not depend on the choice of the primitive root (see. \cite{tsc1}), and the graph may be viewed as the orbital graph of the subgroup the affine group generated by translations and multiplication by $\omega^3$. In \cite{praeger} such graphs appear as cyclotomic partition $Cyc(q,k)$ of $K_q$ and $k=3$. In general, graphs corresponding to such partitions are color-transitive. The case specified here ($k=3$, $p= 2(\bmod~3)$, $r$ even) is the only one with $k>2$, when they are color-symmetric.

Following \cite{praeger} we define also generalized Paley graphs with more colors. Namely, for each $k>2$ and $q=1(\bmod k)$, and such that either $p=2$ or $(q-1)/k$ is even we define $GP_k(q)$ as $k$-colored graph whose vertex set is $F_{q}$, and the edge between two distinct elements has color $i$ ($i=0,\ldots,k-1$) if their difference is of the form $\omega^{km+i}$. This correspond to cyclotomic partitions $Cyc(q,k)$ in \cite{praeger}. We note that our usage of the term "generalized Paley graph" is slightly different: we mean the resulting colored graph, while Li at al. \cite{praeger} mean the simple Caley graph occurring as the factor of the partition. We needed to write a special computer program to make sure that for $k>3$ there is no TSC-graph in this family (see Section~\ref{scomp}). 

\subsection{Graphs determined by directions in vector space.}

We consider now finite vector spaces $F_q^d$ constructed
from finite fields $F_q$. For $d > 1$, by ${\mathcal F}_k(q^d)$ we denote a $k$-colored graph defined on
$F_q^d$, with $k= (q^d-1)/(q-1)$, whose colors are determined naturally by $k$
independent directions in the space. Scalar multiplication and addition
(translations) preserve colors and move vector $(0,v)$ to any vector in the
direction generated by $(0,v)$. This shows that ${\mathcal F}_k(q^d)$ is edge-transitive.
Linear automorphisms of $F_q^d$ act transitively on directions, which shows
that ${\mathcal F}_k(q^d)$ is color-transitive.

One may check directly that three of these graphs, ${\mathcal F}_5(4^2)$, ${\mathcal F}_4(3^2)$, ${\mathcal F}_3(2^2)$ are color-symmetric. The last graph is isomorphic to $GP_3(2^4)$ and its  automorphism group is abstractly isomorphic to the Klein four-group. This is the only known permutation group that is closed but is not the relational group (see. \cite{DS, Kis}; note that Corollary~5.3 in \cite{DS} has wrong proof).

\subsection{Exceptional 3-colored TSC-graphs.} 

There are two further 3-colored TSC-graphs defined in a similar way on vector spaces $F_5^2$ and $F_{11}^2$. The vertex set of $\G_3(5^2)$ is $V= F_5^2$. It has 6 directions determined by vectors starting in $(0,0)$ and ending in $(1,0)$, $(0,1)$, $(1,1)$, $(2,1)$, $(3,1)$, $(4,1)$, The edges in directions determined by $(1,0)$ and $(0,1)$ have color 0, those in directions determined by $(1,1)$ and $(2,1)$ have color 1, and the remaining have color 2. In other words, the graph is determined by the partition:
$$[(1,0), (0,1)], \; \; [(1,1), (2,1)], \; \; [(3,1), (4,1)].$$

Graph $\G_3(11^2)$ has an analogous definition. The vertex set of $\G_3(11^2)$ is $V= F_11^2$. It has $12$ directions determined by vectors starting in $(0,0)$ and ending in $(1,0)$ or $(i,1)$ for $i=0,\ldots 10$. The graph is then determined by the partition
$$[(1,0), (0,1), (1,1), (10,1)], \; \; [(2,1), (3,1), (5,1), (7,1)], \; \; [(4,1), (6,1), (8,1), (9,1)].$$
The fact that these graphs are color-symmetric is proved in the Section~\ref{s3}, where also the automorphism groups of these graphs are presented.

\subsection{Colored Hamming graph.} 

Let us consider the exceptional affine 2-transitive group $M \leq AGL_3(4)$ with $M_0 = 2^{1+4} $ given by Hering's Theorem in \cite[Table 10]{Liebeck}. Let ${\mathcal H}_5(3^4)$ be the colored graph determined by the orbitals of this group. Then combining Theorem~1.1 with Proposition~5.9 of \cite{praeger} we see that ${\mathcal H}_5(3^4)$ is a 5-colored TCS-graph. The corresponding factorization is considered in detail in \cite{Lim}. In particular, it is observed that factors are isomorphic with Hamming graph $H(2,9)$ (both in \cite{Lim} and \cite{praeger} the notation $H(9,2)$ is used, but this seems to be a mistake).

\section{Proof}\label{sproof}

We start from recalling the results we apply in the sequel. First, let us recall
the results of \cite{tsc1} summarized suitably as follows:

\begin{Theorem} \label{twtsc1}
If $\G$ is a $k$-colored TSC-graph then $k\leq 5$ and $\G$ is arc transitive. In addition, for $k>2$ we have the following.
\begin{enumerate}
   \item[\rm (i)]
${\mathcal F}_5(4^2)$ is the unique $5$-colored TSC-graph on $16$ vertices.
Except, possibly, for  $n= 2^8, 3^4$ or $7^4$ there is no other $5$-colored TSC-graph on $n$ vertices.
\item[\rm (ii)] ${\mathcal F}_5(3^2)$ is the unique $4$-colored TSC-graph on $9$ vertices.
Except, possibly, for  $n= 3^4$ there is no other $4$-colored TSC-graph on $n$ vertices.
\item[\rm (iii)] Except for a known infinite family of $3$-colored TSC-graphs $($generalized Paley graphs$)$, there are only finitely many other $3$-colored TSC-graphs with the number of vertices belonging to the set $\{ 2^4, 2^6, 5^2, 11^2, 17^2, 23^2, 89^2\}$.
\end{enumerate}
\end{Theorem}

Combining the above with \cite[Theorem~1.1]{praeger} (and careful inspecting Tables~2 and 3 of \cite{praeger}), we see that almost all TSC-graphs $\G$ have $Aut(\G)$ contained in one-dimensional semilinear affine group. 

\begin{Corollary}\label{pre}
If $\G$ is a $k$-colored TSC-graph mentioned in the theorem above, then one of the following holds
\begin{enumerate}
\item[\rm (i)] $k=5$ and $\G = {\mathcal H}_5(3^4)$,
\item[\rm (ii)] $k=3$ and $n =5^2$ or $n=11^2$, or else
\item[\rm (iii)] $Aut(\G)$ is an affine group contained in $A\Gamma L_1(n)$ and there exists larger group $M \leq A\Gamma L_1(n)$ such that $Aut(\G) \leq M \leq Ext(\G)$, and $M$ permutes transitively colors of $\G$ $($that is, orbits of $Aut(\G)$$)$.
\end{enumerate}
\end{Corollary}

Let $\G$ be a TSC-graph satisfying (iii). Since $Aut(\G)$ contains translations, we may restrict to stabilizers of zero: $A=Aut(\G)_0$ and $M_0$. Then $A \leq M_0 \leq \Gamma L_1(n)$. This makes possible to apply Foulser's description of one-dimensional semilinear groups \cite{Foulser2,FK}. Let us recall briefly the facts we shall need.

Let $\omega$ be a primitive root of the underlying field $F_q$ ($q = p^r$), and at the same time, let it denotes the scalar multiplication by $\omega$, and let $\alpha$ be the generating field automorphism $\alpha: x \rightarrow x^p$.
Then $\Gamma L_1(p^r) = \langle \omega, \alpha\rangle$  is the semidirect product of the normal subgroup $\langle \omega\rangle$ by the subgroup $\langle \alpha \rangle$.
In particular, every element of $g \in \Gamma L_1(p^r)$ has a unique presentation as $g = \omega^e \alpha^s$ for some $0 \leq e < p^r - 1, 0 \leq s < r$. (Alternatively, one may take $0 <  s \leq r$ which is more suitable for the lemma below).

\begin{Lemma} \label{lemf}\cite[Lemma~2.1]{FK} 
Let $H$ be a subgroup of $\langle\omega, \alpha\rangle$.
Then $H$ has the form $H = \langle \omega^d, \omega^e\alpha^s\rangle$, where $d, e, s$ can be chosen to satisfy the following conditions
\begin{enumerate}
   \item[\rm (i)] $s > 0$ and $s|r$;
\item[\rm (ii)] $d > 0$ and $d|(p^r - 1)$;
\item[\rm (iii)] $0 \leq e < d$ and $d|e(p^r-1)/(p^s-1)$.
\end{enumerate}
\end{Lemma}

Moreover, integers $d, e, s$ satisfying the conditions above are unique, and the presentation $H = \langle \omega^d, \omega^e \alpha^s\rangle$ is called the standard form.
In addition, the proof shows that each of $d$ and $s < r$ is in fact the least positive integer such that $H = \langle \omega^d, \omega^e\alpha^s\rangle$ for any $d, e, s$. Moreover, $H$ is a subdirect product of the normal subgroup $\langle \omega^d\rangle$ by the (another cyclic)
subgroup $\langle \omega^e \alpha^s\rangle$.
The cardinality $|H| = (p^r-1)r/ds$.
We note also that $$\alpha \omega = \omega^p\alpha.$$

In our situation, $A$ acts on the set $\{\omega^0,\ldots,\omega^{p^r-1}\}$ of nonzero elements of the field $F_q$. It has $k$ equal orbits corresponding to orbitals of $Aut(\G)$, representing colors of $\G$. Thus we may consider the nonzero elements of $F_q$ as colored: the color of $\omega^i \in F_q$ is the color of the edge $(0,\omega^i)$. We will refer to such an edge as the edge $\omega^i$. By Foulser's description,
$A = \langle \omega^d, \omega^e \alpha^s\rangle$ $M_0 = \langle \omega^{d_1}, \omega^{e_1} \alpha^{s_1}\rangle$, where $d_1|d$, and if $s_1>0$, then $s_1|s$. Moreover if $s>0$ then $s_1>0$.

In the following sections we show that these conditions are very strong and leave little space for existence of suitable objects. The description presented in the main theorem follows directly from the lemmas in the subsequent sections combined with Theorem~\ref{twtsc1} and the Peisert's result in \cite{Peisert}.

\section{TSC-graphs with 4 colors} \label{s4}

We start from the case (ii) of Theorem~\ref{twtsc1}. We prove the following.

\begin{Lemma}
There is no $4$-colored TSC-graph on $3^4$ vertices.
\end{Lemma}

\begin{proof}
Suppose, to the contrary, that $\G$ is a $4$-color TSC-graph on $3^4$ vertices.
By Corollary~\ref{pre} $Aut(\G)$ is an affine group satisfying conditions given in (iii). 
We apply the notation given above. In this case $k=4, n=3^4, p=3, r=4$. In particular, $\alpha: x \to x^3$.

By cardinality formula 
$|A|=80\cdot 4/ds$. Since $A$ is transitive on each color, $({80}/{4})$ divides $|A|$. It follows that $ds|16$. Since, $A$ has four orbits, $d \geq 4$.
Thus, we have three cases to consider.

\emph{Case} 1: $d=16$ and $s=1$.

Here $A = \langle \omega^{16}, \omega^e\alpha \rangle$. Permutation $\omega^{16}$ has $16$ orbits represented by elements $\omega_i$. Since $A$ has 4 equal orbits, for each such an orbit, permutation $\omega^e\alpha$ should "glue together" 4 $\omega^{16}$-orbits, that is, it should act on $\omega^{16}$-orbits as a product of 4 cycles of length $4$. It means, that for each $i=0,1,\ldots,15$, the 4 consecutive images of element $\omega^i$ by permutation $\omega^e\alpha$ should belong to 4 different orbits.
We compute two of these images: $\omega^e\alpha(\omega^i) = \omega^{3i+e}$ and $(\omega^e\alpha)^2 (\omega^i) = \omega^{9i+4e}$.

By Lemma~\ref{lemf}(iii), it follows that $e$ must be divisible by $2$, so $e=2f$ for some $f$. Thus,  $(\omega^e\alpha)^2 (\omega^i) = \omega^{9i+8f}$. It should belong to a different  on $\omega^{16}$-orbit than $\omega^i$ itself. It means that $9i+8f$ is different from $i$ modulo $16$. Applied for $i=0,1$ it means that (modulo $16$) $8f$ is different from $0$, and $9+8f$ is different from $1$, which is impossible.

\emph{Case} 2: $d=8$ and $s\in \{1,2\}$.

Assume first that $s=1$. In this case $\langle \omega^{8} \rangle$ has $8$ orbits, and permutation $\omega^e\alpha$ should "glue together" 4 pairs of such orbits. The consecutive images of $\omega^0$ by $\omega^e\alpha$ are $\omega^e$ and $\omega^{4e}$, which implies that $e=2f$ for some $f$.

Now, the images  $\omega^e\alpha(\omega^i) = \omega^{3i+2f}$ should be in a different $\omega^8$-orbit than $\omega^i$, for each $i$, which means that $3i+2f$ is different from $i$ modulo $8$. It follows that $2i+2f$ is different from 0 modulo 8, that is, $i+f$ is different from 0 modulo 4, for each $i$. Yet, for each $f=0,1,2,3$, there exists $i$ such that $i+f = 0$ modulo $4$, a contradiction.

If $s=2$, then by Lemma~\ref{lemf}(3), $e=0$ or $e=4$. Also, observe that $\alpha^2: x \to x^9$ preserves the orbits of $\omega^8$. Hence $\alpha^2 \in A$. Now, $A \neq \langle \omega^{8}, \alpha^2 \rangle$, since the latter has $8$ orbits. On the other hand, $\alpha^2 \notin  \langle \omega^{8}, \omega^4\alpha^2 \rangle$, a contradiction.

\emph{Case} 3: $d=4$. As above, we see that $\alpha^2$ preserves the orbits of $\omega^4$. Since no other permutation $\omega^e\alpha^s$ preserves these orbits, $A = \langle \omega^4, \alpha^2\rangle$. Here, we need a deeper argument, exploiting total symmetricity. Since $\G$ is a $4$-color TSC-graph, by identifying any two pairs of colors we should obtain $2$-color TSC-graphs that are isomorphic.
Yet, identifying color $0$ with color $2$ and color $1$ with color $3$, we obtain a Paley graph $P(3^4)$, while identifying color $0$ with color $1$ and color $2$ with color $3$, we obtain a Peisert graph $P^*(3^4)$. These graphs are not isomorphic, and thus we have a contradiction.

\end{proof}

\section{TSC-graphs with 5 colors} \label{s5}

In a quite similar way we consider now the case (i) of Theorem~\ref{twtsc1} with $k=5$.

\begin{Lemma} \label{lem74}
There is at most one $5$-colored TSC-graph on $7^4$ vertices.
\end{Lemma}

\begin{proof}
Let $\G$ be a $4$-color TSC-graph on $7^4$ vertices.
By Corollary~\ref{pre} $Aut(\G)$ is an affine group satisfying conditions given in (iii). 
In this case $k=5,  p=7$, and $r=4$. The cardinality $|A|=(7^4-1)\cdot 4/ds$. Since $Aut(\G)$ is arc transitive, $(7^4-1)/{5}$ divides $|A|$. It follows that $ds|20$. Since, $A$ has five orbits, $d \geq 5$.

\emph{Case} 1: $d=20$ and $s=1$.
Here $A = \langle \omega^{20}, \omega^e\alpha \rangle$. Permutation $\omega^{20}$ has $20$ orbits,  and $\omega^e\alpha$ should act as a product of 5 cycles of length $4$ on $\omega^{20}$-orbits. We compute: $\omega^e\alpha(\omega^i) = \omega^{7i+e}$ and $(\omega^e\alpha)^2 (\omega^i) = \omega^{49i+8e}$. The value $49i+8f$ should be different from $i$ modulo $16$, which means that $9i+8f$ is different from $i$, that is $8i+8f$ is different form 0 modulo 20. Consequently, for some $f$, and for all $i$, $i+f$ should be different from 0 modulo 5, which is impossible.

\emph{Case} 2: $d=10$ and $s\in \{1,2\}$.
Now $\omega^e\alpha$ should act as a product of 5 transpositions on $\omega^{10}$-orbits. Hence,
for $s=1$, $(\omega^e\alpha)^2 (\omega^i) = \omega^{49i+8e}$ is the identity on these orbits. Yet, for $i=0,1$ we have $9i+8e = 8e$  and $9+8e$, respectively. It follows that $5$ divides $e$ and, consequently, $9+8\cdot 5f$ is equal 1 modulo $10$, for some $f$, which is a contradiction.

For $s=2$, $\omega^e\alpha^2 (\omega^i) = \omega^{49i+e}$. For $e$ odd, this yields a required product of transpositions, so we need a deeper argument to get a contradiction in this case. 

We make use of the larger group $M$ in Corollary~\ref{pre}(3). Since $M_0$ permutes the colors of $\G$ transitively, and $n=7$, $M_0$ contains in particular a cyclic permutation of colors. Without loss of generality we may assume that $M$ is an extensions of $Aut(\G)$ by a single permutation $c$ permuting colors cyclically. We have  
$M_0 = \langle \omega^{d_1}, \omega^{e_1} \alpha^{s_1}\rangle$, where $d_1|10$, and the index $[M_0:A] = 5$. The only possibility is $M_0 = \langle \omega^{2}, \omega^{e} \alpha^{2}\rangle$. It follows that $\omega^2$ is a cyclic permutation of colors of order $5$. 

Now, the images of $\omega^0$ and $\omega^2$ by $\omega^e\alpha^2$ are, respectively, in $\omega^{10}$-orbits represented by $\omega^e$ and $\omega^{e+8}$. In, particular the corresponding pairs have the same colors. It follows, that $\omega^2\omega^e = \omega^{e+2}$ should have the same color as $\omega^{e+8}$. This contradicts the fact that $\omega^2$ is a cyclic permutation of colors of order $5$.

\emph{Case} 3. $d=5$ and $s\in \{1,2,4\}$. Here $\omega^5$-orbits have to correspond to 5 colors. We check that $\omega^e\alpha^s$ does not preserve $\omega^5$-orbits unless $e=0$ and $s=4$. Thus, $A=\langle\omega^5\rangle$. Note, that in this case, $\omega$ permutes $\omega^5$-orbits cyclically, so we cannot obtain a contradiction with the methods applied so far. 
Adding to $A$ translations, we obtain a group whose orbitals form a 5-colored graph. This is just the only exception pointed out in the formulation of the theorem.
\end{proof}

We observe that this exceptional graph is generalized Paley graph $GP_5(7^4)$ defined in Section~\ref{sdef}. One may check that the whole group  $A \Gamma L_1(7^4)$  preserves colors of $GP_5(7^4)$, which suggests that it may be a TSC-graph with the automorphism group containing $A \Gamma L_1(7^4)$ (and contained in $A \Gamma L_4(7)$). Note that we are able easily to compute this graph and store in computer memory, but since it has $2401$ vertices, computing its automorphism group is beyond the capabilities of modern computer technology. Only combining suitably the computational power with the knowledge we possess makes possible to settle the case. This will be done in the next section.

\begin{Lemma}
There is at least one and at most two $5$-colored TSC-graph on $3^4$ vertices.
\end{Lemma}

\begin{proof}
One of such graphs is ${\mathcal H}_5(3^4)$ described in Section~\ref{sdef}. 
By Corollary~\ref{pre}, for every other $5$-colored TSC-graph $\G$ on $3^4$, $Aut(\G)$ is an affine group satisfying conditions given in (iii). 

In this case $k=5,  p=3$, and $r=4$. The cardinality $|A|=(3^4-1)\cdot 4/ds$. Since $Aut(\G)$ is arc transitive, $(3^4-1)/{5}$ divides $|A|$. It follows that $ds|20$. Since, $A$ has five orbits, $d \geq 5$. The remaining of the proof is essentially the same as that for $n=7^4$.
We leave it to the reader. Again, similarly as in the previous case, the only unsettled case is $A=\langle\omega^5\rangle$ acting on nonzero elements of $F_{81}$.

%The only difference is in Case~1., where $(\omega^e\alpha)^2=\omega^{8e}\alpha^2$. \aaa sprawdzic czy dobrze 

\end{proof}

As before, the exceptional graph is generalized Paley graph, $GP_5(3^4)$, and one may check that the whole group $A \Gamma L_1(3^4)$ preserves colors. So, we need other methods to check this case. Here, $GP_5(3^4)$ has "only" 80 vertices, so one could try to check it using existing computation tools for permutation groups. Yet, we will do it more efficiently applying the same approach as in the case of $GP_5(7^4)$. This is presented in the next section. 
The last case of Theorem~\ref{twtsc1}(i) to consider is that for $n=2^8$.

\begin{Lemma}
There is at most one $5$-colored TSC-graph on $2^8$ vertices.
\end{Lemma}

\begin{proof}
If $\G$ be a $5$-color TSC-graph on $2^8$ vertices, then by
Corollary~\ref{pre}, $Aut(\G)$ is an affine group satisfying conditions given in (iii). 
In this case $k=5,  p=2$, and $r=8$. Using cardinality arguments we have that 
$|A|=(2^8-1)\cdot 8/ds$, and by arc transitivity, $(2^8-1)/{5}$ divides $|A|$.
It follows that $ds|40$. Since, $A$ has five orbits, $d \geq 5$, and since (by Lemma~\ref{lemf}(ii)) $d$ divides $2^8-1 = 255=5 \cdot 51$, it follows that we have only one possibility $d=5$.

Again we check that for $A$ to have 5 orbits, we need to have $e=0$ and $s=0$ or $4$. Since $\alpha^4: x \to x^{16}$ preserves $\omega^5$-orbits, it follows that the only possibility is 
$A = \langle \omega^5, \alpha^4 \rangle$, and the possible exception mentioned in the formulation of the lemma is, again, the generalized Paley graph $GP_5(2^8)$.
\end{proof}

\section{TSC-graphs with 3 colors}\label{s3}

Before we deal with three unsettled cases in the previous section, we complete our investigation considering the case $k=3$. In the three lemmas below we use the fact that, by Corollary~\ref{pre}, $Aut(\G)$ is an affine group satisfying conditions given in (iii). As before we use the notation of Section~\ref{sproof}.

\begin{Lemma}
There is exactly one $3$-colored TSC-graph on $2^4$ vertices.
\end{Lemma}

\begin{proof}
In this case $k=3, n=2^4, p=2, r=4$. In particular, $\alpha: x \to x^2$.
The cardinality formula yields $|A|=(2^4-1)\cdot 4/{ds}$.
Since $Aut(\G)$ is arc transitive, $(2^4-1)/{3}$ divides $|A|$.
It follows that $ds$ divides $12$. By Lemma~\ref{lemf}(ii), $d$ divides $2^4-1 = 15$.
Hence $d=3$.
The only possibility for $A$ to have three orbits is $e=0$ and $s=2$. Since $\alpha^2$ preserves $\omega^3$-orbits, this leads to generalized Paley graph $GP(2^4)$. 
\end{proof}

\begin{Lemma}
There is exactly one $3$-colored TSC-graph on $2^6$ vertices.
\end{Lemma}

\begin{proof}
This case is very similar to the previous one. We have $k=3, p=2$, and $r=6$.  
We have $|A|=(2^6-1)\cdot 4/{ds}$ and 
$(2^6-1)/{3}$ divides $|A|$. Whence $ds|24$. By Lemma~\ref{lemf}(ii), $d$ divides $2^6-1 = 63$, 
yielding $d=3$.
The only possibility for $A$ to have three orbits is $e=0$ and $s\in \{2,4\}$. The latter, $s=4$, is excluded by Lemma~\ref{lemf}(i), since $4$ does not divide $r=6$. Hence, 
$A= \langle \omega^3, \alpha^2 \rangle$, which leads to generalized Paley graph $GP(2^6)$. 
\end{proof}

\begin{Lemma}
There is exactly one $3$-colored TSC-graph on $n = 17^2$, $23^2$, $89^2$ vertices.
\end{Lemma}

\begin{proof}
Let $p \in \{17,23,89\}$. If $s=2$, then $A= \langle \omega^3 \rangle$, and the graph is $GP(p^2)$. The remaining case is $s=1$. 

Here, again, we make use of the facts that $|A|=2(p^2-1)/{d}$ and $(p^2-1)/{3}$ divides $|A|$.
Hence $d|6$. Since $d \geq 3$, we have two cases $d=6$ or $d=3$. The proof is the same in each case $p=17,23,89$, since in each case $p=5$ modulo 6 (in particular, $p=2$ modulo 3). 

\emph{Case} 1. $d=3$. Since $A$ should have exactly 3 orbits, $\omega^e\alpha$  should preserve 
$\omega^3$-orbits. Yet, the image of $\omega^i$ by $\omega^e\alpha$ is $\omega^j$, where $j = pi+e = 2i+e$ modulo 3. It follows that $i=2i+e$, and consequently, $i=-e$ modulo 3. This should be satisfied for each $i$ and a field $e$, a contradiction.

\emph{Case} 2. $d=6$. Here, we look for $e$ such that $\omega^e\alpha$ acts as a product of three transpositions on $\omega^6$-orbits. Now, the image of $\omega^i$ by $\omega^e\alpha$ is $\omega^{j}$, where $j = e+5i = e-i$ modulo $6$. It follows that $A$ has exactly three orbits only when $e$ is odd.

Here we need again a deeper argument, analogous to that applied in Case~2 of Lemma~\ref{lem74}. In the same way we infer that there exists a group  $M_0 = \langle \omega^{d_1}, \omega^{e_1} \alpha^{s_1}\rangle$, where $d_1|6$, and the index $[M_0:A] = 3$. The only possibility is $M_0 = \langle \omega^{2}, \omega^{e} \alpha\rangle$. It follows that $\omega^2$ is a cyclic permutation of colors of order $3$. 

Now, the images of $\omega^0, \omega^1$ and $\omega^2$ by $\omega^e\alpha$ are, respectively, in $\omega^{3}$-orbits represented by $\omega^e,\omega^{e+2}$ and $\omega^{e+1}$. 
This contradicts the fact that $\omega^2$ is a cyclic permutation of colors of order $3$.

\end{proof}

It remains to consider exceptional cases of $n = 5^2$ and $11^2$. We consider each of these cases separately, but before, we establish a more general result we need here.
By Theorem~2.1 of \cite{tsc1} we know that if $\G$ is a $k$-colored TSC-graph then $Aut(\G)$ is an affine group. It follows we may speak of the (finite) vector space $V$ associated with $\G$, and consequently, of sets of vertices forming \emph{lines} in $V$. We prove that for $k>2$ lines are monochromatic in the following sense.

\begin{Lemma}\label{linie}
If $\G$ is a $3$-colored TSC-graph with $k>2$, and $V$ a vector space associated with $\G$, then for each one-dimensional subspace $L$ of $V$, if $v,u\in L$, $v,u \neq 0$, then the edges $(0,v)$ and $(0,u)$ have the same color.
\end{Lemma}

\begin{proof}
We make use of the fact that by the proof of Theorem~2.1 of \cite{tsc1} not only $Aut(\G)$, but  
also $Ext(\G)$ is an affine group (see also \cite[Theorem~15]{Sib}). This means that $Ext(\G)_0 \leq GL_r(p)$, where $|V|=p^r$. In particular, permutations in $Ext(\G)$ preserve lines. 

Let $L =\{0, x_1,\ldots, x_{p-1}\}$, and $x_1$ (that is, $(0,x_1)$) has color $0$. Let $f\in Ext(\G)_0$ be a permutation of vertices that is a transposition of colors $1$ and $2$. Since color $0$ is fixed, and $Aut(\G)$ is transitive on each color, we may assume that $f$ fixes $x_1$, as well. Consequently, $f(L) = L$. 

Now assume that there is a vertex $x_i$ in $L$  is colored $1$. It follows that the number of vertices $x_i \in L$ colored $1$ is the same as that colored $2$. Since the choice of colors is arbitrary, it follows that all the colors are represented in $L$ in the same number. This contradicts the fact that, by Theorem~\ref{twtsc1}, $p = 2(\bmod~3)$.
\end{proof}

%% Generalization dla dowolnego $k$; Linie musza być Paleyami, a dla p-prime, mono --> G(23^2$ i H(2.9) moga bys zdefiniowane na V z podaniem ktore linie jednakowe (tak jak G(11^2) \aaa

\begin{Lemma}
There are exactly two nonisomorphic $3$-colored TSC-graphs on $5^2$ vertices. 
\end{Lemma} 

\begin{proof} 
Let $ \G$ be a $3$-colored TSC-graph on $5^2$ vertices. 
We first construct the field $F_{25}$ taking $2+x^2$ as an irreducible polynomial over $F_5$, and $\omega = 1+x$ as a primitive root. Then we have the natural injection of $\Gamma L_1(25)$ into $GL_2(5)$ given by
$$ \omega \longrightarrow 
\left( 
\begin{array}{cc}
1 &  3   \\ 1 & 1
\end{array}\right), \hspace{1cm}
\alpha \longrightarrow 
\left( 
\begin{array}{cc}
1 &  0  \\ 0 & -1
\end{array}
\right)$$

The associated vector space $V=F_5^2$ has six lines (one-dimensional subspaces) determined by vectors $(1,0)$, $(0,1)$, $(1,1)$, $(2,1)$, $(3,1)$, $(4,1)$. Taking powers of $\omega = 1+x$ modulo $2+x^2$, we check that
they correspond to the lines containing $1$, $\omega^3$, $\omega$, $\omega^2$, $\omega^4$, and $\omega^5$, respectively. By Lemma~\ref{linie}, each line is monochromatic, and consequently, two lines correspond to each color. Without loss of generality we may assume that lines of $(1,0)$ and $(0,1)$ have the same color. (This is so because changing a base by conjugation preserve lines, and therefore it is enough to consider only graphs, in a fixed presentation, with base lines of $(1,0)$ and $(0,1)$ having the same color). Then one of the lines generated by $(2,1)$, $(3,1)$, $(4,1)$ has to have the same color as $(1,1)$. This implies that we have at most three nonisomorphic $3$-colored TSC-graphs on $5^2$. 

The second possibility (lines of $(3,1)$ and $(1,1)$ have the same color) leads to a graph $\G_2$ whose $Aut(\G_2)_0 = \langle\omega^3\rangle$. This group is isomorphic to $Z_8$ and $\G = GP_3(5^2)$.
The first  possibility (lines of $(2,1)$ and $(1,1)$ have the same color) leads to a graph $\G_1$ whose $Aut(\G_1)_0 = \langle\omega^6, \omega^3 \alpha\rangle$. It is not difficult to see that this graph is isomorphic to $\G_1$. Indeed the permutation of $F_5^2$ corresponding to the transposition $(2,3)$ of the underlined field $F_5$ yields the desired isomorphism.   

The last possibility (lines of $(4,1)$ and $(1,1)$ have the same color) leads to another $3$-colored TSC-graph $\G_3$ on $5^2$ vertices. It is straightforward to check that in this case 
$Aut(\G_3)_0$ is generated by matrices  $$\left( \begin{array}{cc}2& 0\\ 0 & 2\end{array}\right),
\left( \begin{array}{cc}0& 1\\ 1 & 0\end{array}\right),
\left( \begin{array}{cc}-1& 0\\ 0 & 1\end{array}\right),$$  which is isomorphic to $D_4 \times Z_2$, has order 16, and cannot be embedded in $\Gamma L_1(5^2)$. Obviously, $\G_3 = \G_3(5^2)$ defined in Section~\ref{sdef}.
\end{proof}

%% Generalization 

\begin{Lemma}
There are exactly two non isomorphic $3$-colored TSC-graphs on $11^2$ vertices. 
\end{Lemma}
 
\begin{proof} In this case the fact that there are at most two $3$-colored TSC-graphs on $11^2$ vertices has been established in the proof of \cite[Theorem~5.1]{tsc1}. The second case in that proof is $Aut(\G)_0 = \langle \omega^3,\alpha^2\rangle < \Gamma L_1(11^2)$ (note that in this case $\alpha^2$ is an identity). This leads to generalized Paley graph $GP(11^2)$. Below, we will use the fact that, in this case, $Aut(\G)_0$ is cyclic and isomorphic to $Z_{40}$ 

The other case in the proof of \cite[Theorem~5.1]{tsc1} is $Aut(\G)_0 = \langle \omega^6, \omega^3\alpha\rangle$ (where $\omega^3$ may be replaced by $\omega$ or $\omega^5$ leading to isomorphic graphs). 
This group is also of order $40$ but is isomorphic with the group $\langle a, b | a^{20} = e, b^2 = a^2, ba = a^{11}b\rangle$, which is not abelian and therefore not isomorphic to ${Z}_{40}$. It follows that the graph $\G$ determined by the orbitals of the corresponding group is not isomorphic to $GP(11^2)$. We need still to prove that it is totally symmetric.

In order to construct the field $F_{11^2}$ we take the polynomial $1+x^2$, which is irreducible over $F_{11}$. As a primitive root we take $\omega = 6 + 2x$. We have twelve lines which contain $\omega^0, \ldots, \omega^{11}$, respectively. For each $i$, we compute vector $(x,y)$ belonging to the same line as $\omega^i$, and using $Aut(\G)_0$ we determine lines with the same color. The results are presented in Table~\ref{tab1}. As we see, this defines graph $\G_3(11^2)$ introduced in Section~\ref{sdef}.    
\begin{table}[h]
\begin{center}
\begin{tabular}{|*{5}{c|c||}c|c|} \hline
$i$ & vector & \textcolor{red}{$i$} & \textcolor{red}{vector} & \textcolor{red}{$i$} & \textcolor{red}{vector} & $i$ & vector & \textcolor{blue}{$i$} & \textcolor{blue}{vector} & \textcolor{blue}{$i$} & \textcolor{blue}{vector}  \\ \hline
0 & (1,0) & \textcolor{red}{1} & \textcolor{red}{(3,1)} & \textcolor{red}{2} & \textcolor{red}{(5,1)} & 3 & (10,1) & \textcolor{blue}{4} & \textcolor{blue}{(9,1)} & \textcolor{blue}{5} & \textcolor{blue}{(4,1)}  \\ \hline 
6 & (0,1) & \textcolor{red}{7} & \textcolor{red}{(7,1)} & \textcolor{red}{8} & \textcolor{red}{(2,1)} & 9 & (1,1) & \textcolor{blue}{10} & \textcolor{blue}{(6,1)} & \textcolor{blue}{11} & \textcolor{blue}{(8,1)} \\ \hline
\multicolumn{2}{|c||}{color 0} &
\multicolumn{2}{|c||}{\textcolor{red}{color 1}} &
\multicolumn{2}{|c||}{\textcolor{red}{color 1}} &
\multicolumn{2}{|c||}{color 0} &
\multicolumn{2}{|c||}{\textcolor{blue}{color 2}} &
\multicolumn{2}{c|}{\textcolor{blue}{color 2}} \\ \hline
\end{tabular} 
\vspace{2mm}
\caption{Correspondence between lines of $F_{11}^2$ and elements $\omega^i$.} \label{tab1}
\end{center}
\end{table}
%\vspace{-5mm} 
Using the table we may check that the graph is totally symmetric.
To this end it is enough to check that the matrix
$$\left(\begin{array}{cc}
1 & 0      \\ 0 & -1
\end{array}        \right)$$
exchanges colors $1$ and $2$, 
while the matrix
$$\left(\begin{array}{cc}
2 & 1      \\ 1 & 4
\end{array}        \right)$$
exchanges colors $0$ and $1$. 
\end{proof}

\section{Computations}\label{scomp}

In order to check whether the exceptional graphs mentioned in Section~\ref{s5} are totally symmetric we wrote a dedicated computer program that made an intensive use of the facts about the structure of the graphs we have established. Below, we present the result and the details of the computations.

\begin{Theorem} None of the graphs $GP_5(3^4)$, $GP_5(7^4)$, $GP_5(2^8)$ is totally symmetric.
\end{Theorem}

The general idea of computations is the same in each case $p^r = 3^4, 7^4, 2^8$. First, we use the fact mentioned in the proof of Lemma~\ref{linie} that the extended automorphism group of a totally symmetric graph is contained in $AGL_r(p)$. It follows that we may restrict to stabilizers of 0 anyway, and our aim is to prove that $GL_r(p)$ does not permute colors in the symmetric way. Colors of $GP_5(p^r)$ corresponds to the orbits of $A = \langle \omega^5 \rangle$.  To fix notation, let us assign color $i$ to the orbit of $\omega^i$ for $i=0,1,\ldots,4$. It is enough to show that, for instance, no permutation in $GL_r(p)$ transposes colors $0$ and $1$. We make use of the fact that permutations in $GL_r(p)$ can be represented by suitable $r\times r$ matrices. We have, respectively, $3^{16}$, $7^{16}$, and $2^{64}$ matrices to check, which are still too large numbers. So we need to use further facts in order to reduce these numbers. Since now the details of the program differ in each case, we describe them, first, for $n=7^4$. 

First, we construct a concrete field on $n=7^4$ elements, using the polynomial $x^4+x^3+x^2+3$ irreducible over $F_7$. We check that $\omega = x$ is a primitive root of $F_n$, generating the multiplicative group. Computing all the powers $\omega^i$ we establish the colors of all vectors $(a_0,a_1,a_2,a_3) = a_0+a_1x+a_2x^2+a_3X^3$. In order to prove that the graph determined by the colored vectors is not totally symmetric it is enough to show that there exists a permutation of colors that is not preserved by any linear transformation. For technical reason we demonstrate this for transposition $(1,2)$ of colors $1$ and $2$. We note that the since the powers $\omega^0, \omega^1,\omega^2$ are the vectors $(1,0,0,0), (0,1,0,0)$ and $(0,0,1,0)$, respectively, their colors are $0,1$ and $2$, respectively.

Thus, we are looking for a $4\times 4$ matrix $B=(a_{ij}), a_{ij}\in F_7$ such that for each vector $v\in F_n$, $v$ and $Bv$ have the same colors, except for that if $v$ has color $1$ then $Bv$ has color 2, $v$ has color $2$ then $Bv$ has color 1. Note that since $A$ is transitive on each color, if there is $B$ with the required properties, then there must be one that in addition fixes $\omega^0 = (1,0,0,0)$. In other words, we may assume that the first column of $B$ is just $(1,0,0,0)$. Further way to reduce the number of matrices to check is to observe that, since the color of $(0,1,0,0)$ is $1$, the second column must be a vector of color 2. Similarly, the third vector must be of color $1$, while the fourth column must be a vector of color $3$. 
Since there are $(7^4-1)/5 = 480$ vectors of each color we have $480^3 = 11\times 10^7$ matrices to generate and check. We wrote a suitable program in C++ and it took some 20 minutes on our PC-computer with a 2GHz processor to get the answer that no matrix satisfies these conditions.

The analogous program for the case of $n=3^4$ has obtained the same answer checking $4096$ matrices in an instant time. The case $n=2^8$ is the hardest one. Here we have $(2^8-1)/5=51$ vectors for each color, and consequently, $51^7 = 9\times 10^{11}$ matrices to check. To reduce this number we make use of the fact that candidates for each column may be further restricted as follows. The sum of the first and the $i$-th column, $i>1$, must be of the same color as the vector $(1,0,\ldots,0,1,\ldots,0)$ with the second $1$ on the $i$-th place. The sum in question may be obtained just by switching the first bite in the vector representing the $i$-th column. Computations in the field $F_{2^8}$ may be also speeded up due to the fact that the underlined field is the so called \emph{Rijndael field} \cite{CL}, where each vector may be treated as a \emph{byte}, and consequently bitwise operations may be applied directly. In particular, computing the image of a vector by a matrix reduces to computing the exclusive-OR operation on the set of column-bytes. We were able to reduce the computation time to a few hours to get the answer that no matrix satisfies the required conditions.

Since the answers in all the three cases was negative, we modified the program slightly to make sure it gives correct answers to other related questions. In particular, our program found correctly the matrices permuting colors in a cyclic way and computed the correct cardinalities of the automorphism groups.

\section*{acknowledgement} We thank Kornel Kisielewicz who helped us to write suitable programs to compute unsettled cases $n=3^4, 7^4$ and $2^8$.

\end{document}